\documentclass[reqno, 14pt]{amsart}
\usepackage{amssymb,graphicx}
\usepackage{amsmath,tikz-cd}
\usepackage{amscd}
\usepackage{amsfonts}
\usepackage{mathrsfs}
\usepackage{stmaryrd, xcolor}
\usepackage{txfonts}
\usepackage{bbm}
\newtheorem{theorem}{Theorem}[section]
\newtheorem{prop}[theorem]{Proposition}
\newtheorem{defn}[theorem]{Definition}
\newtheorem{lemma}[theorem]{Lemma}
\newtheorem{coro}[theorem]{Corollary}
\newtheorem{prop-def}{Proposition-Definition}[section]

\newtheorem{remark}[theorem]{Remark}

\newtheorem{exam}[theorem]{Example}

\begin{document}
\setlength{\oddsidemargin}{0cm} \setlength{\evensidemargin}{0cm}

\title{anti-associative dendriform  algebras}

\author{Zafar Normatov}

\address{ School of Mathematics, Jilin University, Changchun, 130012, China,\\
Institute of Mathematics, Uzbekistan Academy of Sciences, University Street, 9, Olmazor district, Tashkent,
100174, Uzbekistan} \email{z.normatov@inbox.ru}

\def\shorttitle{anti-associative dendriform  algebras}

\begin{abstract}

The general operadic approach to splitting algebraic operations was developed in \cite{BBGN}. By splitting the product in a given algebraic variety $\mathcal{C}$, notion of $\mathcal{C}$-dendriform algebras was systematically studied in \cite{OPV}. This article aims to study ``anti-associative dendriform  algebras", which offer an approach to addressing anti-associativity. These algebras are defined by two operations whose sum is anti-associative. Furthermore, the notion of $\mathcal{O}$-operators on
anti-associative algebras is presented as a tool to interpret anti-associative dendriform algebras. Moreover, anti-associative algebras with
nondegenerate Connes cocycles admit compatible anti-associative dendriform algebra structures.
\end{abstract}



\maketitle

%\tableofcontents \setcounter{section}{0}

\baselineskip=15.8pt

\section{Motivation}

Lie algebras were introduced by Sophus Lie in the 1870s to formalize the study of infinitesimal transformations. Since then, they have become fundamental tools in a wide range of areas, including representation theory, differential geometry, algebraic geometry, number theory, quantum mechanics, and gauge theory.

A vector space $A$ with a bilinear operation $\ast$ is called a Lie admissible algebra if $(A, [-,-])$ is a Lie algebra, where $[-,-]$ is the commutator. In particular, an
{\it associative algebra} $(A, \cdot)$ which is a vector space $A$ with a bilinear operation $\cdot$  satisfiying $x\cdot(y\cdot z)=(x\cdot y)\cdot z$ is a Lie admissible algebra. Hence, there exist a functor from the category $\mathrm{\textbf{As}}$ of associative algebras to the category $\mathrm{\textbf{Lie}}$ of Lie algebras, i.e. 
\[
\mathrm{\textbf{As}}\rightarrow \mathrm{\textbf{Lie}}.
\]
Furthermore, pre-Lie algebras, also known by various names such as left-symmetric algebras, quasi-associative algebras, or Vinberg algebras, are also Lie-admissible algebras. These algebras were first introduced by A. Cayley in 1896 as a type of rooted tree algebra. Hence, one can define a functor from the category $\mathrm{\textbf{Pre-Lie}}$ of all pre-Lie algebras to the category $\mathrm{\textbf{Lie}}$ of all Lie algebras. Thus, one has the following diagram
\begin{equation}\label{half-CD}
\begin{matrix} {}
&\stackrel{}{} & \mathrm{\textbf{Pre-Lie}} \cr {}
&&\downarrow  \cr
 \mathrm{\textbf{As}} &\stackrel{}{\rightarrow}
& \mathrm{\textbf{Lie}}\cr
\end{matrix}\end{equation}

\textbf{Question.} Is it possible to extend the diagram \eqref{half-CD} to obtain a commutative diagram?

To give an affirmative answer one should put the category $\mathrm{\textbf{Dend}}$ of all dendriform algebras on the left-top side (see \cite{A,C,Ron}) i.e.
\begin{equation*}
\begin{tikzcd}
	{\mathrm{\textbf{Dend}}} && {\mathrm{\textbf{Pre-Lie}}} \\
	\\
	{\mathrm{\textbf{As}}} && {\mathrm{\textbf{Lie}}}
	\arrow[ from=1-1, to=1-3]
	\arrow[ from=1-1, to=3-1]
	\arrow[from=1-3, to=3-3]
	\arrow[from=3-1, to=3-3]
\end{tikzcd}
\end{equation*}

The notion of dendriform algebras was introduced by Loday \cite{L1}
in 1995 with the motivation from algebraic $K$-theory and has been
studied quite extensively with connections to several areas in
mathematics and physics, including operads, homology, Hopf algebras,
Lie and Leibniz algebras, combinatorics, arithmetic and quantum
field theory and so on (see \cite{EMP} and the references therein).

Recall that a
{\it dendriform algebra} $(A, \prec, \succ)$ is a vector space $A$
with two binary operations denoted by $\prec$ and $\succ $
satisfying (for any $x,y,z\in A$)
\begin{equation*}(x\prec y)\prec z=x\prec (y\cdot z),\;\;(x\succ y)\prec
z=x\succ (y\prec z),\;\;x\succ (y\succ z)=(x\cdot y)\succ
z,\end{equation*} where $x\cdot y=x\prec y+x\succ y$. Note that $(A,\cdot)$ is
an associative algebra as a direct consequence and the notion of dendriform algebras can be introduced as an approach of ``splitting the associativity".

The central idea of the paper \cite{OPV} was to generalize dendriform algebras by requiring that the operations 
$\prec$ and $\succ$ induce on 
$A\times A$ an algebra structure lying in a chosen category $\mathcal{C}$. The authors introduced the notion of $\mathcal{C}$-dendriform algebras which can be applied to any variety $\mathcal{C}$. This provides a unified framework in which usual dendriform algebras of of J.L-Loday appear as a special case when $\mathcal{C}$ is the category of associative algebras.

The goal of this paper is to study $\mathcal{C}$-dendriform algebras in the case $\mathcal{C}$ is the category of anti-associative algebras i.e. anti-associative dendriform algebras. This characterization parallels that of dendriform algebras.

An anti-associative algebra $(A, \ast)$ is a nonassociative algebra whose multiplication satisfies
the identity 
\[
a\ast (b\ast c)+(a\ast b)\ast c=0, \quad \forall a,b,c\in A.
\]
On the one hand, if we consider the commutator, then the resulting algebra will not be a Lie algebra. But considering the anti-commutator $[a,b]=a\ast b+b\ast a$ derives Jacobi-Jordan algebras (see \cite{Remm}), which seem to be
similar to Lie algebras at first sight, but these algebras are quite different \cite{BF}. This class of algebras appeared in the literature under different
names from different viewpoints, for example ``Lie-Jordan algebras" \cite{KO}, ``mock-Lie algebras" \cite{Zu}. Thus, one has the functor from the category $\mathrm{\textbf{AAs}}$ of anti-associative algebras to the category $\mathrm{\textbf{J-J}}$ of Jacobi-Jordan algebras i.e.
\[
\mathrm{\textbf{AAs}}\rightarrow \mathrm{\textbf{J-J}}
\]
On the other hand, similar to the construction of pre-Lie algebras, pre-Jacobi-Jordan algebras are constructed in \cite{BBMM} as a left-anti-symmetric algebra with respect to anti-associator. Hence, we have

\begin{equation}\label{half-CD2}
\begin{matrix} {}
&\stackrel{}{} & \mathrm{\textbf{Pre-J-J}} \cr {}
&&\downarrow  \cr
 \mathrm{\textbf{AAs}} &\stackrel{}{\rightarrow}
& \mathrm{\textbf{J-J}}\cr
\end{matrix}\end{equation}

It is then natural to ask whether the diagram  \eqref{half-CD2} can be extended to a commutative diagram. In this paper we give an affirmative answer to this question and the category $\mathrm{\textbf{AAs-Dend}}$ of all \textbf{anti-associative dendriform  algebras} play a left-top category in the commutative diagram i.e.
\begin{equation}\label{CD1}
\begin{tikzcd}
	{\mathrm{\textbf{AAs-Dend}}} && {\mathrm{\textbf{Pre-J-J}}} \\
	\\
	{\mathrm{\textbf{AAs}}} && {\mathrm{\textbf{J-J}}}
	\arrow[dashed, from=1-1, to=1-3]
	\arrow[dashed, from=1-1, to=3-1]
	\arrow[from=1-3, to=3-3]
	\arrow[from=3-1, to=3-3]
\end{tikzcd}
\end{equation}

The classification of any class of algebras is a fundamental but highly challenging problem. It is often one of the first issues encountered when attempting to understand the structure of a particular class of algebras. Specifically, the classification of algebras of dimension 
$n$ (up to isomorphism) within a variety defined by a given family of polynomial identities is a classic question in the theory of non-associative algebras. Numerous studies have addressed the classification of low-dimensional algebras across various varieties of both associative and non-associative algebras \cite{AANS, BM, FKK, K, KSTT, P, RRB1, RRB2}.

The paper is organized as follows. In Section 2, we recall the notion of anti-associative dendriform  algebras as an approach of splitting the anti-associativity and prove that the diagram \eqref{CD1} is commutative. The notions of $\mathcal{O}$-operators and
anti-Rota--Baxter operators on anti-associative algebras are introduced to interpret anti-associative dendriform  algebras in Section 3. In Section 4, the relationships between anti-associative dendriform  algebras and  Connes cocycles on
anti-associative algebras are given. In the last Section 5, we introduce various notions of nilpotency for anti-associative dendriform  algebras and demonstrate that these definitions are equivalent. Moreover, we give the classification of $2$-dimensional anti-associative dendriform  algebras over the field of complex numbers.

\section{anti-associative dendriform  algebras}

In this section, we prove that the diagram \eqref{CD1} is commutative. Moreover, using a bimodule structure on a vector space, we define an anti-associative algebra structure on the direct sum of the vector space with itself by using the operations of an anti-associative dendriform algebra.
 
%An anti-associative algebra is a nonassociative algebra whose multiplication satisfies the identity $a(bc)+ (ab)c = 0$.\cite{Remm} %Such algebras are nilpotent.\cite{Remm} 

%\begin{prop}[\cite{Remm}]
%An anti-associative algebra is nilpotent of nilindex 4.
%\end{prop}

\begin{defn}[\cite{BF}]
An algebra $(A,\diamond)$ over a field $\mathbb{F}$ is called a Jacobi–Jordan algebra if it satisfies
the following two identities
\begin{align*}
x \diamond y -y \diamond x = 0,\\
x \diamond (y \diamond z) + y \diamond (z \diamond x) + z \diamond (x \diamond y) = 0,
\end{align*}
for all $x, y, z \in A$.
\end{defn}
In other words, Jacobi–Jordan algebras are commutative algebras which satisfy the
Jacobi identity. In the case of Lie algebras, the product is not commutative but anticommutative.

Recall that $(A,\circ)$ is called a \textbf{Jacobi-Jordan-admissible
algebra} \cite{BBMM, Remm}, where $A$ is a vector space with a bilinear operation
$\circ: A\otimes A\rightarrow A$, if the bilinear operation
$[-,-]:A\otimes A\rightarrow A$ defined by
\begin{equation*} [x,y]=x\circ y+y\circ x,\;\; \forall x,y\in A,
\end{equation*}
makes $(A,[-,-])$ a Jacobi-Jordan algebra. In this case,  $(A,[-,-])$ is called the \textbf{sub-adjacent Jacobi-Jordan algebra} of 
$(A,\circ)$ and denoted by $(\mathfrak{J}(A),[-,-])$. Obviously, an anti-associative algebra is a Jacobi-Jordan-admissible algebra.

Another example of Jacobi-Jordan-admissible algebras is  pre-Jacobi-Jordan algebras.

\begin{defn}\cite{BBMM}
%A pre-Jacobi-Jordan algebra $A$ is a vector space  with a bilinear multiplication $\cdot: A\otimes A\rightarrow A$ satisfying
%\[
%(xy)z + x(yz)+ (yx)z + y(xz)=0, \quad  \forall x, y, z \in A.
%\]
A pre-Jacobi-Jordan algebra is a vector space $A$ together with a bilinear multiplication $\cdot: A \otimes A \to  A$, such that for any $x,y,z\in A$, the anti-associator $AAs(x,y,z):=(x\cdot y)\cdot z+x\cdot (y\cdot z)$ is skew-symmetric in $x,y,$ i.e. $AAs(x,y,z)=-AAs(y,x,z)$ or equivalently
\begin{eqnarray*}
(x\cdot y)\cdot z+x\cdot (y\cdot z)=-(y\cdot x)\cdot z-y\cdot(x\cdot z).
\end{eqnarray*}

\end{defn}

We now give a definition analogous to that of an associative admissible algebra \cite{GLB}.

\begin{defn}
Let $A$ be a vector space with two bilinear operations
\[
\succ: A \otimes A \rightarrow A, \quad  \prec: A \otimes A \rightarrow A.
\]
Define a bilinear operation $\ast$ as
\begin{equation}\label{cdot}
x \ast y = x \succ y + x \prec y, \quad \forall x, y \in A.
\end{equation}
The triple $(A, \succ, \prec)$ is called an anti-associative admissible algebra if $(A, \ast)$ is an anti-associative algebra.
In this case, $(A, \ast)$ is called the associated anti-associative algebra of $(A, \succ, \prec$).
\end{defn}

\begin{remark}
The triple $(A, \succ, \prec)$ is an anti-associative admissible algebra if and only if the following
equation holds:
\begin{equation}\label{anti-associative admissible}
\begin{split}
&(x \succ y) \succ z + (x \prec y) \succ z + (x \succ y) \prec z + (x \prec y) \prec z\\
&= -x \succ (y \succ z) - x \succ (y \prec z) - x \prec (y \succ z) - x \prec (y \prec z), \quad \forall x, y, z \in A.   
\end{split}
\end{equation}

\end{remark}

\begin{defn}\label{defi:anti-associative dendriform  algebras}
Let $A$ be a vector space with two bilinear operations $\succ$ and $\prec$. The triple
$(A, \succ, \prec)$ is called a \textbf{anti-associative dendriform  algebra} if the following equations hold:
\begin{align}\label{def:id1}
(x\ast y)\succ z&=-x\succ(y\succ z),\\\label{def:id2}
x\prec(y\ast z)&=-(x\prec y)\prec z,\\\label{def:id3}
x\succ(y\prec z)&=-(x\succ y)\prec z, \quad \forall x,y,z\in A,
\end{align}
where the bilinear operation $\ast$ is defined by Eq. \eqref{cdot}.
\end{defn}

\begin{exam}%\label{one-dim}
Let $(A, \prec, \succ)$ be a one-dimensional anti-associative dendriform algebra with basis element $e$.
Assume that $e \succ e = \alpha e$ and
$e \prec e = \beta e$, where $\alpha, \beta \in  \mathbb{F}$.
Then, by equations \eqref{def:id1} -- \eqref{def:id3}, we obtain
$\alpha(2\alpha + \beta) = 0$ and $\beta(2\beta + \alpha) = 0$. Hence $\alpha = \beta = 0$, and every one-
dimensional anti-associative dendriform algebra is trivial.
\end{exam}

\begin{theorem}\label{property-1}
Let $(A, \succ, \prec)$ be an anti-associative dendriform  algebra. Then the followings hold.

(i). Define a bilinear operation $\ast$ by Eq.~\eqref{cdot}. Then $(A,\ast)$ is an anti-associative algebra,
 called the {\bf associated anti-associative algebra} of $(A, \succ, \prec).$
 Furthermore, $(A, \succ, \prec)$ is called a {\bf compatible anti-associative dendriform  algebra structure} on $(A,\ast).$

(ii). The bilinear operation $\circ:A\otimes A\rightarrow A$ given by \begin{equation*}
x\circ y=x\succ y+y\prec x,\ \ \forall x,y\in A,
\end{equation*}
  defines a pre-Jacobi-Jordan algebra, called the {\bf associated pre-Jacobi-Jordan algebra} of $(A, \succ, \prec).$
 
(iii). Both $(A, \ast)$ and $(A, \circ)$ have the same sub-adjacent Jacobi-Jordan algebra $(\mathfrak{g}(A),[-,-])$ defined by
\begin{equation*}[x, y]=x\succ y+x\prec
y+y\succ x+y\prec x,\quad \forall x,y\in
A.\end{equation*}
Moreover, the commutative diagram (2) holds.
\end{theorem}

\begin{proof}

(i) is straightforward.

(ii) Let $x, y, z \in A$. Then we have
\begin{equation*}
\begin{split}
(x \circ y) \circ z=&(x\succ y+y\prec x)\succ z+z\prec (x\succ y+y\prec x)\\
=&(x\succ y)\succ z+(y\prec x)\succ z+z\prec (x\succ y)+z\prec (y\prec x),\\
x \circ (y \circ z)=&x\succ(y\succ z+z\prec y)+(y\succ z+z\prec y)\prec x\\
=&x\succ(y\succ z)+x\succ(z\prec y)+(y\succ z)\prec x+(z\prec y)\prec x.
\end{split}
\end{equation*}
By swapping $x$ and $y$, we have
\begin{equation*}
\begin{split}
(y \circ x) \circ z=(y\succ x)\succ z+(x\prec y)\succ z+z\prec (y\succ x)+z\prec (x\prec y),\\
y \circ (x \circ z)=y\succ(x\succ z)+y\succ(z\prec x)+(x\succ z)\prec y+(z\prec x)\prec y.
\end{split}
\end{equation*}
Using Eqs.~\eqref{def:id1}--\eqref{def:id3}, we obtain
\begin{equation*}
\begin{split}
x\circ(y\circ z)+y\circ(x\circ z)=&x\succ(y\succ z)+x\succ(z\prec y)+(y\succ z)\prec x+(z\prec y)\prec x\\
&+y\succ(x\succ z)+y\succ(z\prec x)+(x\succ z)\prec y+(z\prec x)\prec y\\
=&x\succ(y\succ z)+(z\prec y)\prec x+y\succ(x\succ z)+(z\prec x)\prec y\\
=&-(x\succ y+x\prec y)\succ  z-z\prec(y\succ x+y\prec x)\\
&-(y\succ x+y\prec x)\succ z-z\prec(x\succ y+x\prec y)\\
=&-(y\circ x)\circ z-(x\circ y)\circ z.
\end{split}
\end{equation*}
Moreover, we have
\[
x\circ y+y\circ x=x \succ y + y\prec x+ x\prec y + y\succ x  = x \ast y + y \ast x, \quad \forall  x, y \in A.
\]
Thus $(A,\circ)$ is a Jacobi-Jordan-admissible algebra and hence a pre-Jacobi-Jordan algebra.

(iii). It is straightforward. 
Note that it also appears in the proof of
(ii).
\end{proof}

Let $(A,\ast)$ be an anti-associative algebra. Recall that a {\bf bimodule} of $(A,\ast)$ is a triple $(V, l, r)$ consisting of a vector space $V$ and linear maps
$l,r: A\rightarrow \rm{End}_\mathbb{F}(V)$ such that
$$l(x\ast y)v=-l(x)(l(y)v), \ \ r(x\ast y)v=-r(y)(r(x)v), \ \ l(x)(r(y)v)=-r(y)(l(x)v),\ \ \forall x,y\in A, v\in V.$$
In particular, $(A, L_\ast, R_\ast)$ is a bimodule of $(A,\ast)$, where $L_\ast,R_\ast:A\rightarrow {\rm End}_\mathbb{F}(A)$ are two linear maps defined by
$L_\ast(x)(y)=R_\ast(y)(x)=x\ast y$ for all $x,y\in A$ respectively.

Suppose that $(A,\ast)$ is an anti-associative algebra. Let $V$ be a vector space and $l,r: A\rightarrow {\rm End}_{\mathbb F}(V)$ be linear maps.
Then $(V, l, r)$ is a bimodule of $(A,\ast)$ if and only if there is an anti-associative algebra structure on the direct sum $A\oplus V$ of vector spaces  with the following bilinear operation, still denoted by~$\ast$:
$$(x,u)\ast(y,v)=(x\ast y,l(x)v+r(y)u),\quad \forall x,y\in A,u,v\in V.$$
We denote this anti-associative algebra by $A\ltimes_{l,r}V$.

\begin{theorem}\label{anti-associative dendriform  algebra equivalent}
Let $A$ be a vector space with two bilinear operations $\succ$ and $\prec$. Define a bilinear operation $\ast$ by Eq.~\eqref{cdot}.
Then $(A, \succ, \prec)$ is an anti-associative dendriform  algebra if and only if $(A,\ast)$ is an anti-associative algebra and $(A, L_\succ, R_\prec)$ is a bimodule of $(A,\ast)$, where
two linear maps $L_\succ,R_\prec:A\rightarrow {\rm End}_\mathbb{F}(A)$ are defined by
$$L_\succ(x)(y)=x\succ y,\quad R_\prec(x)(y)=y\prec x,\quad \forall x,y\in A.$$

\end{theorem}

\begin{proof}

From Eqs.~\eqref{anti-associative admissible},~\eqref{def:id1},~\eqref{def:id2} and ~\eqref{def:id3} it follows that $(A, \succ, \prec)$ is an anti-associative dendriform  algebra if and only if $(A, \succ, \prec)$ is an anti-associative admissible algebra, that is, $(A, \ast)$ is an anti-associative algebra, and for all $x, y, z \in A$, the following equations hold:
  \begin{equation*}
 x\succ(y\succ z)=-(x\cdot y)\succ z,\ (x\prec y)\prec z=-x\prec(y\cdot z),\
  (x\succ y)\prec z=-x\succ (y\prec z).
  \end{equation*}

Let $x,y,z\in A$. Then we have
\begin{align*}
   L_\succ(x\ast y)(z)=-L_\succ(x)(L_\succ(y)z)\ \ &\Longleftrightarrow \ \ x\succ (y\succ z)=-(x\ast y)\succ z, \\
  R_\prec(x\ast y)(z)=-R_\prec(y)(R_\prec(x)z)\ \ &\Longleftrightarrow \ \ (z\prec x)\prec y=-z\prec(x\ast y), \\
  L_\succ(x)(R_\prec(y)z)=-R_\prec(y))(L_\succ(x)z) \ \ &\Longleftrightarrow \ \ x\succ(z\prec y)=-(x\succ z)\prec y.
\end{align*}
Hence $(A, L_\succ, R_\prec)$ is a bimodule.
\end{proof}

Let $V$ be a vector space and $V^*$ be a dual space of $V$.
For $\phi\in V^*$ and $u\in V$, we write $\langle \phi, u\rangle:=\phi(u)$.

 The following conclusion is obvious.
\begin{lemma}\label{lem2.9} Let $V$ be a vector space and $V^*$ be a dual space of $V$.
Let $(V,l, r)$ be a bimodule of an anti-associative algebra $A$. 
\begin{itemize}
\item[(i)] Let $l^{*}, r^{*}: A \rightarrow \mathfrak{gl}(V^{*})$ be the linear maps given by
\[
\langle l^{*}(x)u^{*}, v\rangle = \langle u^{*}, l(x)v \rangle, \ \langle r^{*}(x)u^{*},v\rangle = \langle u^{*}, r(x)v\rangle, \quad  \forall x \in A, u^{*} \in V^{*}, v \in V.
\]
Then $(V^{*},r^{*}, l^{*}) $ is a bimodule of A.

\item[(ii)]  $(V, l,0), (V, 0,r), (V^{*},r^{*}, 0)$ and $(V^{*}, 0,l^{*})$ are bimodules of $A$.
\end{itemize}
\end{lemma}
\begin{proof}
(i) \quad Let $x \in A, u^{*} \in V^{*}, v \in V$. Then
\[
\begin{split}
\langle l^{*}(x\ast y)u^{*}, v\rangle=&\langle u^{*},l(x\ast y)v\rangle=-\langle u^{*}, l(x)(l(y)v) \rangle=-\langle l^{*}(x)u^{*},l(y)v\rangle=-\langle l^{*}(y)(l^{*}(x)u^{*}),v\rangle,\\
\langle r^{*}(x\ast y)u^{*}, v\rangle=&\langle u^{*},r(x\ast y)v\rangle=-\langle u^{*}, r(y)(r(x)v) \rangle=-\langle r^{*}(y)u^{*},r(x)v\rangle=-\langle r^{*}(x)(r^{*}(y)u^{*}),v\rangle,
\end{split}
\]
and
\[
\begin{split}
\langle (l^{*}(x)r^{*}(y))u^{*}, v\rangle=&\langle (l^{*}(x)(r^{*}(y)u^{*}), v\rangle=\langle r^{*}(y)u^{*},l(x)v\rangle=\langle u^{*},r(y)(l(x)v)\rangle\\
=&-\langle u^{*}, l(x)(r(y)v) \rangle=-\langle l^{*}(x)u^{*},r(y)v\rangle=-\langle r^{*}(y)(l^{*}(x)u^{*}),v\rangle,\\
\end{split}
\]
Hence, $(V^{*}, r^{*}, l^{*})$ is a bimodule of $A$.

(ii) is obvious.
\end{proof}

\begin{coro}\label{cor:anti-d} Let $A$ be a vector space with two bilinear operations $\succ,\prec: A\otimes A\rightarrow A$.
Then on the direct sum $\hat A:=A\oplus A$ of vector spaces,
the following bilinear operation
\begin{equation}\label{eq:asso-double}
(x,a)\ast (y,b)=(x\succ y+x\prec y,
x\succ b+a\prec y),\;\;\forall x,y,a,b\in
A,
\end{equation}
makes an anti-associative algebra $(\hat A,\ast)$ if and only if
$(A,\succ,\prec)$ is an anti-associative dendriform  algebra.
\end{coro}
\begin{proof}

The first proof: It is clear that $(\hat A,\ast)$ is an anti-associative algebra if and only if $(A, \succ, \prec)$ is an anti-associative admissible algebra,
that is, $(A,\ast)$ is an associative algebra,
and $(A, L_\succ, R_\prec)$ is a bimodule
of the associated anti-associative algebra, which is equivalent to the
fact that $(A, \succ, \prec)$ is a
anti-associative dendriform  algebra by Theorem \ref{anti-associative dendriform  algebra equivalent}.

The second proof: By the definition of $\mathcal{C}$-dendriform algebras in the case $\mathcal{C}$ is an anti-associative algebra (see \cite{OPV}).
\end{proof}

\section{anti-associative dendriform  algebras and $\mathcal{O}$-operators}

M. Aguiar in 2000 \cite{A} was the first who noticed a relation between Rota--Baxter
algebras and dendriform algebras. The concept of an $\mathcal{O}$-operator, which generalizes the Rota--Baxter operator, was introduced by C. Bai, L. Guo, and K. Ni in 2010 [4]. They established that any dendriform algebra arises explicitly from an algebra admitting an $\mathcal{O}$-operator. By the motivation of their work, V. Yu. Gubarev and P. S. Kolesnikov proved that every dendriform dialgebra in $\operatorname{Var}$ can be embedded into a
Rota--Baxter algebra of weight zero in the same variety \cite{GP}. 

Similarly to the above works, the authors of \cite{OPV} proved the following Proposition.

\begin{prop}\label{Rota-Baxter on C-dendriform}
Let $R$ be a Rota--Baxter operator on an algebra $(A, \mu)$
which belongs to $\mathcal{C}$. For $a, b \in A$, let $a \succ b := R(a)b$ and $a \prec b := aR(b)$.
Then $(A, \succ, \prec)$ is a $\mathcal{C}$-dendriform algebra.
\end{prop}

In this section, for the case where $\mathcal{C}$ is an anti-associative algebra, we prove that the result holds more generally for $\mathcal{O}$-operators , which we substitute for the Rota--Baxter operators in the construction.

\begin{defn}
Let $(A,\ast)$ be an anti-associative algebra and $(V, l, r)$ be a
bimodule. A linear map $T: V\rightarrow A$ is called an {\bf
$\mathcal{O}$-operator} of $(A,\ast)$ associated to $(V, l,
r)$ if the following equation holds: \begin{equation*}
T(u)\ast
T(v)=T\big(l(T(u))v+r(T(v))u\big), \quad \forall u,v\in
V.
\end{equation*}  
In particular, an
$\mathcal{O}$-operator $T$ of $(A,\ast)$ associated to the
bimodule $(A,L_\ast,R_\ast)$ is called an {\bf Rota--Baxter
operator}, that is, $T:A\rightarrow A$ is a linear map satisfying
\begin{equation}\label{eq:ANTIRBO}T(x)\ast T(y)=T\big(T(x)\ast y+x\ast T(y)\big), \quad
\forall x,y\in A.\end{equation} In this case, we call $(A,T)$ a {\bf
Rota--Baxter algebra} .
\end{defn}

\begin{theorem}\label{anti}
Let $(A,\ast)$ be an anti-associative algebra and $(V,l,r)$ be a
bimodule.
 Suppose that $T: V\rightarrow A$ is an $\mathcal{O}$-operator of $(A,\ast)$ associated to $(V, l,
 r)$.
  Define two bilinear operations $\succ,\prec$ on $V$  respectively as
  \begin{align}\label{anti-O1}
  u\succ v=l(T(u))v,\quad u\prec v=r(T(v))u,\quad \forall u,v\in V.
  \end{align}
  Then  $(V, \succ, \prec)$ is an
anti-associative dendriform  algebra. In this
case, $T$ is a homomorphism of anti-associative algebras from the
associated anti-associative algebra $(V,\ast)$ to $(A,\ast)$.
Furthermore, there is an induced anti-associative dendriform  algebra structure
on $T(V)=\{T(u)~|~u\in V\}\subseteq A$ given by
\begin{align}\label{induce}
T(u)\succ T(v)=T(u\succ v),\quad
T(u)\prec T(v)=T(u\prec v),\quad \forall
u,v\in V,
\end{align}
and $T$ is a homomorphism of anti-associative dendriform  algebras.

\end{theorem}

\begin{proof}
Let $u,v,w\in V$. Then we have
\begin{equation*}
\begin{split}
-(u\ast v)\succ w=&-(u\succ v+u\prec v)\succ w=-l(T\big(l(T(u))v\big))w-l(T\big(r(T(v))u\big))w=-l(T(u)\ast T(v))w\\
=&l(T(u))(l(T(v))w)=u\succ(v\succ w).
\end{split}
\end{equation*}
Similarly, we have 
\[
(u\prec v)\prec w=-u\prec(v\ast w), \quad  (u\succ v)\prec w=-u\succ (v\prec w).
\]
Thus by 
Definition~\ref{defi:anti-associative dendriform  algebras}, $(V,
\succ, \prec)$ is an anti-associative dendriform  algebra. 

Moreover,
\begin{align*}
T(u\ast v)&=T(u\succ v+u\prec v)=T(u\succ v)+T(u\prec v) =T\big(l(T(u))v\big)+T\big(r(T(v))u\big)
=T(u)\ast T(v),\ \ \forall u,v\in V,
\end{align*}
shows that $T$ is a homomorphism of anti-associative algebras from the associated anti-associative algebra $(V,\ast)$ to $(A ,\ast)$

Furthermore, for any $u,v,w\in V,$ we can obtain
\begin{align*}
T(u)\succ (T(v)\succ T(w))&=T(u\succ (v\succ w))=-T(u\succ v)\succ T(w)-T(u\prec v)\succ T(w)\\
&=-(T(u)\succ T(v))\succ T(w)-(T(u)\prec T(v))\succ T(w)\\
&=-(T(u)\succ T(v)+T(u)\prec T(v))\succ T(w).
\end{align*}
Similarly, one can derive
\begin{align*}
 (T(u)\prec T(v))\prec T(w)&=-T(u)\prec (T(v)\succ T(w)+T(v)\prec T(w)),\\
 (T(u)\succ T(v))\prec T(w)&=-T(u)\succ (T(v)\prec T(w)),
\end{align*}
which imply that $T(V)$ is an anti-associative dendriform  algebra. This
completes the proof.
\end{proof}

Below, we provide an example to illustrate the construction of an anti-associative dendriform algebra using Proposition \ref{Rota-Baxter on C-dendriform}.

\begin{exam}
Let $(A,\ast)$ be a complex anti-associative  algebra with a basis $e_1,e_2$
whose non-zero products are given by
$$e_1\ast e_1=e_2.$$
Suppose that $R:A\rightarrow A$ is a linear map whose corresponding matrix
is given by $\left( \begin{matrix}\alpha_{11} &\alpha_{12}\cr
\alpha_{21}&\alpha_{22}\cr\end{matrix}\right)$ under the basis $e_1,e_2$.
Applying Eq.~(\ref{eq:ANTIRBO}) to the pair  $\{e_1,e_2\}$ yields $\alpha_{21}=0$. For the pair $\{e_1,e_1\}$, we obtain $\alpha_{11}(\alpha_{11}-2\alpha_{22})=0$. The other pairs yield no further constraints. Thus $R$ is a Rota--Baxter
operator on $(A,\ast)$ if and only if $\alpha_{11}(\alpha_{11}-2\alpha_{22})=0$. Then we have two cases:

\textbf{Case 1.} $\alpha_{11}=0$. Then the corresponding matrix of the Rota--Baxter operator is of the form $\begin{pmatrix}
    0&\alpha_{12}\\
    0&\alpha_{22}
\end{pmatrix}$. Hence, one can obtain the trivial anti-associative dendriform  algebra.

\textbf{Case 2.} $\alpha_{11}\neq0$. Then $\alpha_{11}=2\alpha_{22}$ and the corresponding matrix of the Rota--Baxter operator is of the form $\begin{pmatrix}
    2\alpha_{22}&\alpha_{12}\\
    0&\alpha_{22}
\end{pmatrix}$. Hence, we  get an anti-associative dendriform  algebra whose non-zero products are given by
\[
e_1\succ e_1=2\alpha_{22}e_2, \quad e_1\prec e_1=2\alpha_{22}e_2, \quad \alpha_{11}\in \mathbb{C}^{*}.
\]
By the basis change $e'_1=e_1, \ e'_2=2\alpha_{22}e_2$, it is isomorphic to $Rh_2[1]$ given in Theorem \ref{thm5.8}.
\end{exam}

Next we consider  invertible $\mathcal{O}$-operators.
\begin{theorem}\label{anti-associative dendriform -invertible O}
Let $(A,\ast)$ be an anti-associative algebra. Then there is a
compatible anti-associative dendriform  algebra structure on $(A,\ast)$ if and
only if there exists an invertible $\mathcal{O}$-operator of
$(A,\ast)$.
\end{theorem}

\begin{proof}
Suppose that $(A,\succ,\prec)$ is a compatible
anti-associative dendriform  algebra structure on $(A,\ast)$. Then
$$x\ast
y=x\succ y+x\prec y=L_\succ
(x) y+R_\prec (y)x,\;\;\forall x,y\in A.$$ Hence the
identity map ${\rm Id}:A\rightarrow A$ is an invertible
$\mathcal{O}$-operator of $(A,\ast)$ associated to the
bimodule $(A,L_\succ,R_\prec)$.

Conversely, suppose that $T:V\rightarrow A$ is an invertible
$\mathcal{O}$-operator of $(A,\ast)$ associated to a
bimodule $(V, l, r)$ of $(A,\ast)$. Then by  Proposition \ref{anti}, there exists
a compatible anti-associative dendriform  algebra structures on $V$ and $T(V)=A$ defined by
Eqs.~\eqref{anti-O1} and ~\eqref{induce} respectively. Let $x,y\in
A$. Then there exist $u,v\in V$ such that $x=T(u),y=T(v)$. Hence
we have
\begin{align*}
x\ast y&=T(u)\ast T(v)=T(l(T(u))v+r(T(v))u)=T(u\succ v+u\prec v)\\
&=T(u)\succ T(v)+T(u)\prec T(v)=x\succ y+x\prec y.
\end{align*}
So $(A,\succ,\prec)$ is a compatible
anti-associative dendriform  algebra structure on $(A,\ast)$.
\end{proof}

\begin{prop}\label{embedding}
Let $(A,\ast)$ be an anti-associative algebra and $(V, l, r)$ be a
bimodule. Suppose that $T:V\longrightarrow A$ is a linear map.
Then $T$ is an $\mathcal{O}$-operator of $(A,\ast)$
associated to $(V, l, r)$ if and only if the linear map
$$\hat T: A\ltimes_{l,r} V\longrightarrow A\ltimes_{l,r} V,\quad (x,u)\longmapsto
(T(u),0),$$ is a Rota--Baxter operator on the anti-associative
algebra $A\ltimes_{l,r}V$.

\end{prop}
\begin{proof}
Let  $x,y\in A,u,v\in V$. Then we have
\begin{eqnarray*}
\hat T((x,u))\ast \hat T((y,v))&=&(T(u),0)\ast(T(v),0)=(T(u)\ast T(v),0),\\
\hat T((x,u))\ast(y,v)&=&(T(u),0)\ast(y,v)=(T(u)\ast y, l(T(u))v),\\
(x,u)\ast \hat T ((y,v))&=&(x,u)\ast (T(v),0)=(x\ast T(v),
r(T(v))u).
\end{eqnarray*}
Hence $\hat T$ is a Rota--Baxter operator on the anti-associative
algebra $A\ltimes_{l,r}V$ if and only if
$$(T(u)\ast T(v),0)=\bigg(T\big(l(T(u))v+r(T(v))u\big),0\bigg),$$
that is, $T$ is an $\mathcal{O}$-operator of $(A,\ast)$
associated to $(V, l, r)$.
\end{proof}

\begin{coro}
Let $(A,\succ,\prec)$ be a
anti-associative dendriform  algebra and $(A,\ast)$ be the associated
anti-associative algebra. Set $\hat A=A\oplus A$ as the direct sum of
vector spaces. Define a bilinear operation $\ast$ on $\hat A$ by
Eq.~\eqref{eq:asso-double} and a linear map $\widehat {\rm
Id}:\hat A\rightarrow \hat A$ by 
\begin{equation*}
\widehat {\rm Id}((x,y))=(y,0),\;\;\forall x,y\in A.
\end{equation*}
Then $\widehat {\rm Id}$ is a Rota--Baxter operator on the
anti-associative algebra $(\hat A,\ast)$, that is, $(\hat A, \widehat
{\rm Id})$ is a Rota--Baxter algebra.
\end{coro}

\begin{proof}
By Corollary~\ref{cor:anti-d}, $(\hat A,\ast)$ is an associative
algebra, which is exactly
$A\ltimes_{L_\succ,R_\prec} A$. Since
${\rm Id}:A\rightarrow A$ is an  $\mathcal{O}$-operator of
$(A,\ast)$ associated to the bimodule
$(A,L_\succ,R_\prec)$, by
Proposition~\ref{embedding}, $\widehat {\rm Id}$ is a Rota--Baxter operator on the anti-associative algebra $(\hat
A,\ast)$.
\end{proof}

\section{anti-associative dendriform  algebras and
Connes cocycles}

In this section, we show that anti-associative dendriform  algebras can be obtained from nondegenerate Connes cocycles of anti-associative algebras.

A {\bf Connes cocycle} on an associative algebra $(A,\cdot)$ is an
antisymmetric bilinear form $\mathcal{B}$ satisfying
\begin{align*}
\mathcal{B}(x\cdot y, z)+\mathcal{B}(y\cdot z,
x)+\mathcal{B}(z\cdot x, y)=0,\quad \forall x,y,z\in A.
\end{align*}
It corresponds to the original definition of cyclic cohomology by
Connes (\cite{C}). Note that there is a close relation between
dendriform algebras and Connes cocycles (\cite{B}). Next we
consider a Connes cocycle on anti-associative algebras.

\begin{defn}
Let $(A,\ast)$ be an anti-associative algebra. A bilinear form $\mathcal{B}$ on $A$ is called a {\bf Connes cocycle} if
$$\mathcal{B}(a\ast b,c)+\mathcal{B}(b\ast c,a)+\mathcal{B}(c\ast a,b)=0,\quad \forall a,b,c\in A.$$
\end{defn}

\begin{theorem}\label{compatible structure}
Let $(A,\ast)$ be an anti-associative algebra and $\mathcal{B}$ be a
nondegenerate  Connes cocycle on $(A,\ast)$. Then
there exists a compatible anti-associative dendriform  algebra structure $(A$,
$\succ$, $\prec)$ on $(A,\ast)$ defined by
\begin{align}\label{non-comm}
\mathcal{B}(x\succ y, z)=\mathcal{B}(y,z\ast x),\ \
\mathcal{B}(x\prec y, z)=\mathcal{B}(x,y\ast z),\ \
\forall x,y,z\in A.
\end{align}
\end{theorem}

\begin{proof}
Define a linear map $T: A\rightarrow A^{*}$ by $$\langle T(x),
y\rangle=\mathcal{B}(x,y),\ \ \forall x,y\in A,$$
where, $A^*$ is a dual vector space of $A$. By Lemma \ref{lem2.9}, $(A^{*},R^{*}, L^{*})$ is a bimodule of $A$. Then $T$ is invertible
and $T^{-1}$
is an $\mathcal{O}$-operator of the anti-associative algebra $(A, \ast)$ associated to the bimodule $(A^{*},R^{*}, L^{*})$.
By Theorem \ref{anti-associative dendriform -invertible O}, there is a compatible anti-associative dendriform  algebra structure $\succ, \prec$ on $(A, \ast)$ given by
\[
x\succ y=T^{-1}\big(R^{*}(x)T(y)\big), \quad x\prec y=T^{-1}\big(L^{*}(y)T(x)\big), \quad \forall x,y\in A,
\]
which gives exactly Eq. \eqref{non-comm}.
\end{proof}

Next, we turn to the double construction of Connes cocycles. Let $(A, \ast_{A})$ be an anti-associative
algebra and suppose that there is a anti-associative algebra structure $\ast_{A^{*}}$ on its dual space $A^{*}$. We
construct an anti-associative algebra structure on the direct sum $A \oplus A^{*}$ of the underlying vector
spaces of $A$ and $A^{*}$
such that both $A$ and $A^{*}$ are subalgebras and the symmetric bilinear
form on $A\oplus A^{*}$ given by equation 
\begin{equation*}
\mathcal{B}(x+a^{*},y+b^{*})=\langle x,b^{*}\rangle+\langle
a^{*},y\rangle,\ \  \forall x,y\in A, a^{*},b^{*}\in A^{*}.
\end{equation*}
is a Connes cocycle on $A\oplus A^{*}$. Such a construction
is called a \textbf{double construction} of Connes cocycle associated to $(A, \ast_{A})$ and $(A^{*},\ast_{A^{*}})$ and we
denote it by $(T(A)=A\bowtie A^{*},\mathcal{B})$.

\begin{prop}
Let $(T(A) = A\bowtie A^{*}, \mathcal{B})$ be a double construction of Connes cocycle. Then
there exists a compatible anti-associative dendriform  algebra structure $\succ, \prec$ on $T(A)$ defined by equation \eqref{non-comm}.
Moreover, $A$ and $A^{*}$ are anti-associative dendriform  subalgebras with this product.
\end{prop}
\begin{proof}
The first half follows from Theorem \ref{compatible structure}. Let $x, y \in A$. Set $x \succ y = a+b^{*}$, where $a\in A, \ b^{*}\in A^{*}$. Since $A$ is an anti-associative subalgebra of $T(A)$ and 
\begin{equation}\label{B(A,A)}
\mathcal{B}(A, A) = \mathcal{B}(A^{*}, A^{*})=0,
\end{equation}
we have $\mathcal{B}(b^{*}, A^{*})=0$ and
\[
\mathcal{B}(b^{*}, A) = \mathcal{B}(x \succ y-a, A) \overset{\eqref{B(A,A)}}{=} \mathcal{B}(x \succ y, A)\overset{\eqref{non-comm}}{=} \mathcal{B}(y, A \ast x) \overset{\eqref{B(A,A)}}{=}  0.
\]
Therefore $b^{*}=0$ due to the nondependence of $\mathcal{B}$. Hence $x\succ y = a \in A$. Similarly, $x \prec y \in A$.
Thus $A$ is an anti-associative dendriform  subalgebra of $T(A)$ with the product $\succ, \prec$. By symmetry of $A$ and $A^{*}$,
$A^{*}$ is also an anti-associative dendriform  subalgebra.
\end{proof}

\begin{defn}
Let $(T(A_1) = A_1 \bowtie A_1^{*}, \mathcal{B}_1)$ and $(T(A_2) = A_2 \bowtie A_2^{*}, \mathcal{B}_2)$ be two double constructions of Connes cocycles. They are isomorphic if there exists an isomorphism of anti-associative
algebras $\phi : T(A_1) \rightarrow T(A_2)$ satisfying the following conditions:
\begin{equation}\label{isomorphism}
\phi(A_1) = A_2,\  \phi(A_1^{*}) = A_2^{*}, \  \mathcal{B}_1(x, y) = \phi^{*}\mathcal{B}_2(x, y) = \mathcal{B}_2(\phi(x),\phi(y)), \quad \forall x, y \in A_1.
\end{equation}
\end{defn}

\begin{prop}
Two double constructions of Connes cocycles $(T(A_1) = A_1 \bowtie A_1^{*}, \mathcal{B}_1)$
and $(T(A_2) = A_2 \bowtie A_2^{*}, \mathcal{B}_2)$ are isomorphic if and only if there exists an anti-associative dendriform  algebra
isomorphism $\phi : T(A_1) \rightarrow T(A_2)$ satisfying Eq. \eqref{isomorphism}, where the anti-associative dendriform  algebra
structures on $T(A_1)$ and $T(A_2)$ are given by Eq. \eqref{non-comm}  respectively.
\end{prop}
\begin{proof}
It is straightforward.
\end{proof}

\section{Isomorphism classes of two  dimensional
anti-associative dendriform  algebras}

In this section, we classify two-dimensional anti-associative dendriform algebras. Moreover, we introduce the annihilator of anti-associative dendriform algebras and demonstrate its role in the construction of new algebras from old ones. 

\begin{prop}
Let $(A, \succ, \prec)$
be an anti-associative dendriform  algebra. If one of the
binary operations $\succ, \ \prec$ in $A$ vanishes, then with respect to the second one, the algebra $A$ is anti-associative.
\end{prop}
\begin{proof}
It is straightforward.
\end{proof}

\begin{defn}
An anti-associative dendriform  algebra $(A, \succ, \prec)$ is $2$-nilpotent if $(x*_{i_1} y)*_{i_2} z = x*_{i_3} (y*_{i_4} z) = 0$
for all $*_{i_1}, *_{i_2}, *_{i_3}, *_{i_4}~\in~\{\succ, \prec\}$ and $x, y, z \in A$.
\end{defn}

\begin{prop}\label{2-nilpotent}
Let $(A, \succ, \prec)$ be an anti-associative dendriform algebra satisfying $x \succ y + x \prec y = 0$ for all $x, y \in A$. Then $(A, \succ, \prec$) is a 2-nilpotent algebra.
\end{prop}
\begin{proof}
Let $x, y, z \in A$. Then, we obtain
\[
\begin{split}
&x\succ(y\succ z)\stackrel{\eqref{def:id1}}{=}-(x\succ y+x\prec y)\succ z=0,\\
&(x\prec y)\prec z\stackrel{\eqref{def:id2}}{=}-x\prec(y\succ z+y\prec z)=0,\\
\end{split}
\]
Using these results we can get
\[
\begin{split}
&x\prec(y\succ z)=x\prec(y\succ z)+x\succ(y\succ z)=0,\\
&(x\prec y)\succ z=(x\prec y)\succ z+(x\prec y)\prec z=0.
\end{split}
\]
Furthermore, we have
\[
\begin{split}
&(x\succ y)\succ z=(x\succ y)\succ z+(x\prec y)\succ z=(x\succ y+x\prec y)\succ z=0,\\
&x\prec (y\prec z)=x\prec (y\prec z)+x\prec (y\succ z)=x\prec (y\prec z+y\succ z)=0,\\
&(x\succ y)\prec z=(x\succ y)\prec z+(x\succ y)\succ z=0,\\
&x\succ (y\prec z)\stackrel{\eqref{def:id3}}{=}(x\succ y)\prec z=0.
\end{split}
\]
\end{proof}

Now, we classify anti-associative dendriform  algebra structures on
$2$-dimensional complex vector space.

\begin{theorem}\label{thm5.8}
Any non-trivial two-dimensional anti-associative dendriform  algebra can be
included in one of the following classes of algebras:

$Rh_1:\quad e_1\prec e_1=e_2$

$Rh_2[\lambda]:\quad e_1\succ e_1=e_2, \ e_1\prec e_1=\lambda e_2, \quad \lambda\in \mathbb{C}.$
\end{theorem}

\begin{proof}
Let $A$ be a two dimensional anti-associative dendriform  algebra. Set
\[
\begin{split}
e_i\succ e_j=&\alpha_{ij1}e_1+\alpha_{ij2}e_2,\\
e_i\prec e_j=&\beta_{ij1}e_1+\beta_{ij2}e_2, \quad 1\leq i,j\leq 2.
\end{split}
\]
Now we define $e_i\ast e_j=e_i\succ e_j+e_i\prec e_j$. Then by Theorem \ref{property-1}, the algebra $(A, \ast)$ is anti-associative. As given in Examples in \cite{Remm}, there are only two non-isomorphic anti-associative algebras i.e. abelian algebra $A_1$ and the algebra $A_2$ with non-zero multiplication $e_1\ast e_1=e_2$. 

Let the algebra $(A,\ast)$  be the algebra $A_2$. Assume that $(A, \succ, \prec)$ is a compatible
anti-associative dendriform  algebra structure on $(A, \cdot)$. Set $\alpha_{ij}:=\alpha_{ij1}, \ \beta_{ij}:=\alpha_{ij2}, \ 1\leq i,j\leq 2$.

Then we have
\[
\begin{cases}
e_1\prec e_1=e_2-\alpha_{11}e_1-\beta_{11}e_2, \quad e_1 \prec e_2 = -\alpha_{12}e_1-\beta_{12}e_2,\\
e_2\prec e_1=-\alpha_{21}e_1-\beta_{21}e_2, \quad e_2 \prec e_2 = -\alpha_{22}e_1-\beta_{22}e_2.
\end{cases}
\]

\textbf{Case $\alpha_{22}\neq0$.}
By considering Eqs. \eqref{def:id1} and \eqref{def:id2}  for the triple $\{e_2,e_2,e_2\}$  we obtain $\alpha_{21}=-\beta_{22}, \alpha_{12}=-\beta_{22}$, respectively.

Comparing the coefficients of the basis $e_1$ on the considering Eqs.  \eqref{def:id1} and \eqref{def:id2}  for the triple $\{e_1,e_1,e_1\}$  we obtain $\beta_{22}=0$.

Again, considering Eq. \eqref{def:id1} for the triple $\{e_1,e_1,e_2\}$ we get $\alpha_{22}=0$ which is a contradiction. Hence, there is no such a case.

\textbf{Case $\alpha_{22}=0$.} By considering  Eq. \eqref{def:id1} for the triples $\{e_2,e_2,e_1\}, \ \{e_2,e_2,e_2\}$ and Eq. \eqref{def:id2} for the triple $\{e_1,e_2,e_2\}$ we obtain $\alpha_{21}=0, \beta_{22}=0, \alpha_{12}=0$, respectively.

Again by considering  Eq. \eqref{def:id1} for the triples $\{e_1,e_1,e_1\}, \ \{e_1,e_1,e_2\}$ and Eq. \eqref{def:id2} for the triple $\{e_2,e_1,e_1\}$ we obtain $\alpha_{11}=0, \beta_{12}=0, \beta_{21}=0$, respectively. Then we have the anti-associative dendriform  algebra with non-zero products given by
\[
\begin{cases}
e_1\succ e_1=\gamma e_2\\
e_1\prec e_1=(1-\gamma)e_2
\end{cases}
\]
Then, if $\gamma=0$ we obtain 
\[
Rh(A)_1:\quad e_1\prec e_1=e_2,
\]
if $\gamma\neq0$ by the basis change $e'_1=e_1, e_2'=\gamma e_2$ we obtain
\[
Rh(A)_2[\lambda]:\quad e_1\succ e_1=e_2, \ e_1\prec e_1=\lambda e_2, \quad \lambda\in \mathbb{C}, \ \lambda\neq-1.
\]

Let the algebra $(A,\ast)$  be abelian algebra $A_1$. Then by
Proposition \ref{2-nilpotent} the anti-associative dendriform  algebra $(A, \succ, \prec)$ associated to the abelian anti-associative algebra is $2$-nilpotent. Thus we have
$(x \succ y) \succ z = x \succ (y \succ z) = 0$ which implies that $(A, \succ)$ is an anti-associative $2$-nilpotent algebra. Since all  $2$-dimensional anti-associative algebras are $2$-nilpotent, we get $Rh_2[-1]$.

It is possible to consider the multiplication $\prec$ as above. However, it is not difficult to show that the
constructed algebras are isomorphic. This finishes the proof.
\end{proof}

\begin{defn}
The following sets are called the annihilator of anti-associative and anti-associative dendriform  algebras, respectively:
$${\rm
Ann}_{AAs}(A)=\{x\in A~|~x\ast y=y\ast x=0,\ \forall y\in A\},$$
$${\rm
Ann}_{AAsD}(A)=\{x\in A~|~x\succ y=x\prec y=y\succ x=y\prec x=0,\ \forall y\in A\}.$$
\end{defn}

An ideal $I$ of an anti-associative dendriform  algebra $A$ is a subalgebra of the algebra $A$ that satisfies the
conditions:
$$x\succ y, \ x\prec y, \ y\succ x, \ y\prec x \in I, \ \ \mbox{for any} \ \  x\in A,\  y\in I.$$

It is obvious that the annihilator of an arbitrary anti-associative dendriform  algebra is an ideal.

Now we prove the following Proposition which is a useful tool to construct higher dimensional anti-associative dendriform  algebras.

\begin{prop}\label{quotient prop}
Let $(A, \ast)$ be an anti-associative algebra and let $(A, \succ, \prec)$ be a compatible anti-associative dendriform 
algebra structure on $(A, \ast)$. If the annihilator $(\rm{Ann}_{AAs}(A), \ast)$ of $(A, \ast)$ is the annihilator $(\rm{Ann}_{AAsD}(A), \succ, \prec)$ of $(A, \succ, \prec)$, then
the quotient $(A/\rm{Ann}_{AAsD}(A), \succ, \prec$) is a compatible anti-associative dendriform  algebra structure on $(A/\rm{Ann}_{AAs}(A), \ast)$.
\end{prop}

\begin{proof}
The Proposition can be described as follows
$$
\begin{CD}
(A, \succ, \prec) @>\quad\quad x\ast y=x\succ y+x\prec y\quad\quad>> (A, \ast)  \\
@V{\bar{x}=x+\rm{Ann}_{AAsD}(A)}VV @VV{\bar{x}=x+\rm{Ann}_{AAs}(A)}V  \\
(A/\rm{Ann}_{AAsD}(A), \succ, \prec) @>>\quad\quad\bar{x}\ast \bar{y}=\bar{x}\succ \bar{y}+\bar{x}\prec \bar{y}\quad\quad> (A/\rm{Ann}_{AAs}(A), \ast)
\end{CD}
$$
To show that $(\rm{Ann}_{AAsD}(A), \succ, \prec)$ is a compatible anti-associative dendriform  algebra structure on $(A/\rm{Ann}_{AAs}(A), \ast)$, it is
sufficient to show $\bar{x}\ast \bar{y}=\overline{x\ast y}$. The structure is given by $\bar{x}\ast\bar{y}=\bar{x}\succ \bar{y}+\bar{x}\prec\bar{y}$. Indeed,
$$
\begin{array}{lll}
\bar{x}\ast\bar{y}&=&\bar{x}\succ \bar{y}+\bar{x}\prec\bar{y}\\
&=&(x+\rm{Ann}_{AAsD}(A))\succ(y+\rm{Ann}_{AAsD}(A))+(x+\rm{Ann}_{AAsD}(A))\prec(y+\rm{Ann}_{AAsD}(A))\\
&=&(x\succ y+x\prec y)+\rm{Ann}_{AAs}(A)=\overline{x\ast y}.
\end{array}
$$
\end{proof}

In the following example, we show an application of Proposition \ref{quotient prop}.
\begin{exam}
Let $(A,\ast)$ be an anti-associative algebra with non-zero multiplications: $e_1\ast e_2=-e_2\ast e_1=e_3$. Then $\langle e_3\rangle$ is the annihilator of the algebra $A$. Assume $(\langle e_3 \rangle, \prec, \succ)$ is the annihilator of a three-dimensional anti-associative dendriform algebra $(A, \prec, \succ)$ and $(A/\langle e_3\rangle, \prec, \succ)$ is the algebra $Rh_2[-1]$. Then, according to Proposition \ref{quotient prop}, we can write
\[
\begin{cases}
e_1\succ e_1= e_2+\alpha_1 e_3,\quad  e_1\succ e_2=\alpha_2 e_3, \quad e_2\succ e_1 = \alpha_3 e_3, \quad e_2\succ e_2=\alpha_4e_3\\ e_1\prec e_1=-e_2-\alpha_1 e_3, \quad e_1\prec e_2=(1-\alpha_2) e_3, \quad e_2\prec e_1=-(1+\alpha_3)e_3, \quad e_2\prec e_2=-\alpha_4e_3.
\end{cases}
\]
Using Eqs \eqref{def:id1}-\eqref{def:id3} for the triples $\{e_1, e_1, e_1\}$ and $\{e_2,e_1,e_1\}$ we obtain $\alpha_2=\alpha_4=0, \alpha_3=-1$. Now we consider the basis change $e'_2 = e_2 +\alpha_1e_3$. Then we obtain the algebra
\[
e_1\succ e_1=e_2,\quad   e_2\succ e_1 = -e_3,\quad  e_1\prec e_1=-e_2, \quad e_1\prec e_2=e_3.
\]
\end{exam}

\section*{Acknowledgements}

The author would like to express his gratitude to Yunhe Sheng for many fruitful discussions and
suggestions in the preparation of this article. The author is very grateful to the referee for useful comments and suggestions.

\section{Declarations}

\textbf{Ethical Approval }

not applicable;

\textbf{Funding} 

not applicable;

\textbf{Availability of data and materials}

not applicable:

\end{document}